\documentclass[letterpaper]{amsart}

\usepackage{amssymb,latexsym}
\usepackage{verbatim}
\usepackage{url}
\usepackage[colorlinks,citecolor={blue},linkcolor={blue}]{hyperref}

\theoremstyle{plain}
\newtheorem{theorem}{Theorem}[section]

\newtheorem{proposition}[theorem]{Proposition}

\newtheorem*{theorem*}{Theorem}
\theoremstyle{definition}

\theoremstyle{remark}

\numberwithin{equation}{section}

\newcommand{\Union}{\bigcup}

\newcommand{\spacing}{\hspace{5mm}}

\renewcommand{\bar}{\overline}

\renewcommand{\a}{\alpha}
\renewcommand{\epsilon}{\varepsilon}
\renewcommand{\phi}{\varphi}

\newcommand{\eqstack}[1]{\stackrel{\textup{(#1)}}{=}}

\newcommand{\geqstack}[1]{\stackrel{\textup{(#1)}}{\geq}}

\newcommand{\Ac}{\mathcal{A}}

\newcommand{\X}{\mathcal{X}}

\newcommand{\Normal}{\mathcal{N}}

\DeclareSymbolFont{AMSb}{U}{msb}{m}{n}
\DeclareMathSymbol{\N}{\mathbin}{AMSb}{"4E}
\DeclareMathSymbol{\Z}{\mathbin}{AMSb}{"5A}
\DeclareMathSymbol{\R}{\mathbin}{AMSb}{"52}
\DeclareMathSymbol{\Q}{\mathbin}{AMSb}{"51}
\DeclareMathSymbol{\PP}{\mathbin}{AMSb}{"50}
\renewcommand{\P}{\PP}
\DeclareMathSymbol{\E}{\mathbin}{AMSb}{"45}

\begin{document}

\title[A simple example of DPM inconsistency]{A simple example of Dirichlet process mixture inconsistency for the number of components}
\author[J. W. Miller]{Jeffrey W. Miller}
\author[M. T. Harrison]{Matthew T. Harrison}
\thanks{Division of Applied Mathematics, 
Brown University,
Providence, RI 02912}
\thanks{\url{jeffrey_miller@brown.edu}}

\begin{abstract}
For data assumed to come from a finite mixture with an unknown number of components, it has become common to use Dirichlet process mixtures (DPMs) not only for density estimation, but also for inferences about the number of components. The typical approach is to use the posterior distribution on the number of components occurring so far --- that is, the posterior on the number of clusters in the observed data. However, it turns out that this posterior is not consistent --- it does not converge to the true number of components. In this note, we give an elementary demonstration of this inconsistency in what is perhaps the simplest possible setting: a DPM with normal components of unit variance, applied to data from a ``mixture'' with one standard normal component. Further, we find that this example exhibits severe inconsistency: instead of going to 1, the posterior probability that there is one cluster goes to 0.
\end{abstract}

\date{}

\maketitle

\section{Introduction}

It is well-known that Dirichlet process mixtures (DPMs) of normals are consistent for the density --- that is, given data from a sufficiently regular density $p_0$ the posterior converges to the point mass at $p_0$ (see \cite{Wu_2010,Ghosh_2003} for details and references). However, it is easy to see that this does not imply consistency for the number of components, since for example, a good estimate of the density might include superfluous components having vanishingly small weight.

Despite the fact that a DPM has infinitely many components with probability~$1$, it has become common to apply DPMs to data assumed to come from a finite mixture, and to apply the posterior on the number of components used to generate the observed data (in other words, the posterior on the number of clusters in the data) for inferences about the true number of components (see \cite{Escobar_1995,West_1994,Otranto_2002,Medvedovic_2002,Huelsenbeck_2007,Xing_2006,Lartillot_2004} for a few prominent examples). Thus, it is important to understand the properties of this procedure.

In this note, we give a simple example in which a DPM is applied to data from a finite mixture and the posterior distribution on the number of clusters does not converge to the true number of components. In fact, DPMs exhibit this type of inconsistency under very general conditions, as we will show elsewhere --- however, the aim of this note is brevity and clarity. To this end, we focus our attention on a special case that is as simple as possible: a ``standard normal DPM'', that is, a DPM using univariate normal components of unit variance, with a standard normal base measure (prior on component means).


Some authors have empirically observed that the DPM posterior tends to overestimate the number of components (e.g. \cite{West_1994,Lartillot_2004,Onogi_2011}, among others), and have found that ignoring tiny clusters tends to mitigate this issue. It might be possible to obtain consistent estimators in this way. However, by adopting such a procedure, one is abandoning the DPM model, and it is not clear what model (if any) would give rise to such a procedure.

A more natural way to obtain consistency is the following: if the number of components is unknown, put a prior on the number of components. For example, draw the number of components $s$ from a probability mass function $p(s)$ on $\{1,2,\dotsc\}$ with $p(s)>0$ for all $s$, draw mixing proportions $\pi=(\pi_1,\dotsc,\pi_s)$ from an $s$-dimensional Dirichlet (given $s$), draw component parameters $\theta_1,\dotsc,\theta_s$ i.i.d.\ (given $s$ and $\pi$) from an appropriate prior, and draw $X_1,X_2,\dotsc$ i.i.d.\ (given $s$, $\pi$, and $\theta_{1:s}$) from the resulting mixture. This approach has been widely used \cite{Nobile_1994,Richardson_1997,Green_2001,Nobile_2007}. Strictly speaking, as defined, such a model is not identifiable --- but it is fairly straightforward to modify it to be identifiable by choosing one representative from each equivalence class. Subject to a modification of this sort, it can be shown (see e.g. \cite{Nobile_1994}) that under very general conditions such models are (a.e.)\ consistent for the number of components, the mixing proportions, the component parameters, and the density (for data from a finite mixture of the chosen family). It is a common misperception that efficient (approximate) inference is more difficult in these models than in DPMs --- to the contrary, we have found that an appropriately constructed MCMC sampler for such a model is nearly identical to an MCMC sampler for a DPM. Further details will be provided elsewhere, since they are beyond the scope of this note.

The rest of the paper is organized as follows. In Section \ref{section:setup}, we define the DPM model under consideration. In Section \ref{section:simple}, we give an elementary proof of inconsistency for a standard normal DPM. In Section \ref{section:proof}, we show (using Hoeffding's strong law of large numbers for U-statistics) that this example is in fact severely inconsistent, in the sense that the posterior probability of the true number of components goes to~$0$.

\section{Setup}
\label{section:setup}

In this section, we define the Dirichlet process mixture model. 

\subsection{Dirichlet process mixture model}
\label{section:model}

The DPM model was introduced by Ferguson \cite{Ferguson_1983} for the purpose of Bayesian density estimation, and was made practical through the efforts of several authors (see \cite{Escobar_1998} and references therein).
We will use $p(\cdot)$ to denote probabilities under the DPM model (as opposed to other probability distributions that will be considered in what follows).
The core of the DPM is the so-called Chinese restaurant process (CRP), which defines a certain probability distribution on partitions.
Given $n\in\{1,2,\dotsc\}$ and $t\in\{1,\dotsc,n\}$, let $\Ac_t(n)$ denote the set of all \emph{ordered} partitions $(A_1,\dotsc,A_t)$ of $\{1,\dotsc,n\}$ into $t$ nonempty sets. In other words,
$$\Ac_t(n) =\Big\{(A_1,\dotsc,A_t): A_1,\dotsc,A_t \,\mbox{are disjoint,}\,\,
    \Union_{i = 1}^t A_i =\{1,\dotsc,n\},\,\,|A_i|\geq 1\,\,\forall i\Big\}. $$
The CRP with concentration parameter $\a>0$ defines a probability mass function on $\Ac(n)=\Union_{t = 1}^n \Ac_t(n)$ by setting
$$p(A) =\frac{\a^t}{\a^{(n)}\,t!}\prod_{i = 1}^t (|A_i|-1)!$$
for $A\in\Ac_t(n)$, where $\a^{(n)} =\a(\a +1)\cdots(\a + n -1)$. Note that since $t$ is a function of $A$, we have $p(A) = p(A,t)$. (It is more common to see this distribution defined in terms of unordered partitions $\{A_1,\dotsc,A_t\}$, in which case the $t!$ does not appear in the denominator --- however, for our purposes it is more convenient to use the distribution on ordered partitions $(A_1,\dotsc,A_t)$ obtained by uniformly permuting the parts. This does not affect the prior or posterior on $t$.)

Consider the hierarchical model
\begin{align}
& p(A,t) =p(A) =\frac{\a^t}{\a^{(n)}\,t!}\prod_{i = 1}^t (|A_i|-1)!, \label{equation:partitions}\\
& p(\theta_{1:t} \mid A,t) =\prod_{i = 1}^t p(\theta_i), \,\,\,\mbox{and} \notag\\
& p(x_{1:n} \mid \theta_{1:t},A,t) =\prod_{i = 1}^t \prod_{j\in A_i} p_{\theta_i}(x_j), \notag
\end{align}
where $p(\theta)$ is a prior on component parameters $\theta\in\Theta$, and $\{p_\theta:\theta\in\Theta\}$ is a parametrized family of distributions on $x\in\X$ for the components. Typically, $\X\subset\R^d$ and $\Theta\subset\R^k$ for some $d$ and $k$. Here, $x_{1:n} =(x_1,\dotsc,x_n)$ with $x_i\in\X$, and $\theta_{1:t} =(\theta_1,\dotsc,\theta_t)$ with $\theta_i\in\Theta$. The marginal distribution on $x_{1:n}$ is called a \emph{Dirichlet process mixture (DPM) model}.

The prior on the number of clusters $t$ under this model is $p_n(t) = \sum_{A\in\Ac_t(n)} p(A,t)$. We use $T_n$ (rather than $T$) to denote the random variable representing the number of clusters, as a reminder that its distribution depends on $n$. 
Note the distinction between the terms ``component'' and ``cluster'': a \emph{component} is part of a mixture distribution, while a \emph{cluster} is the set of (indices of) data points coming from a given component.

Since we are concerned with the posterior distribution $p(T_n = t \mid x_{1:n})$ on the number of clusters, we will be especially interested in the marginal distribution on $(x_{1:n},t)$, given by
\begin{align}
p(x_{1:n},T_n = t) & =\sum_{A\in\Ac_t(n)}\int p(x_{1:n},\theta_{1:t},A,t)\,d\theta_{1:t} \notag\\
& =\sum_{A\in\Ac_t(n)} p(A)\prod_{i = 1}^t 
    \int \Big(\prod_{j\in A_i} p_{\theta_i}(x_j)\Big) p(\theta_i)\,d\theta_i \notag\\
& =\sum_{A\in\Ac_t(n)} p(A)\prod_{i = 1}^t m(x_{A_i}) \label{equation:x_t}
\end{align}
where for any subset of indices $S\subset\{1,\dotsc,n\}$, we denote $x_S =(x_j: j\in S)$ and let $m(x_S)$ denote the single-cluster marginal of $x_S$,
\begin{align}
m(x_S) =\int\Big(\prod_{j\in S} p_\theta(x_j)\Big)\,p(\theta)\,d\theta. \label{equation:marginal}
\end{align}

\subsection{Specialization to the standard normal case}
\label{section:specialization}

In this note, for brevity and clarity, we focus on the univariate normal case with unit variance, with a standard normal prior on means --- that is, for $x\in\R$ and $\theta\in\R$,
\begin{align*}
& p_\theta(x) =\Normal(x \mid \theta,1) =\frac{1}{\sqrt{2\pi}}\exp(-\tfrac{1}{2}(x-\theta)^2),\spacing\mbox{and}\\
& p(\theta)=\Normal(\theta \mid 0,1)=\frac{1}{\sqrt{2\pi}}\exp(-\tfrac{1}{2}\theta^2).
\end{align*}
It is a straightforward calculation to show that the single-cluster marginal is then
\begin{align}
m(x_{1:n})
=\frac{1}{\sqrt{n+1}}\,p_0(x_{1:n})\exp\Big(\frac{1}{2}\frac{1}{n+1}\Big(\sum_{j = 1}^n x_j\Big)^2\Big),
\label{equation:normal_marginal} 
\end{align}
where $p_0(x_{1:n}) = p_0(x_1)\cdots p_0(x_n)$ (and $p_0$ is the $\Normal(0,1)$ density).
When $p_\theta(x)$ and $p(\theta)$ are as above, we refer to the resulting DPM as a \emph{standard normal DPM}.



\section{Elementary example of inconsistency}
\label{section:simple}

In this section, we prove the following result, exhibiting a simple example in which a DPM is inconsistent for the number of components: the true number of components is $1$, but the posterior probability of $T_n = 1$ does not converge to $1$. To keep it simple, we set $\a = 1$, but the proof extends trivially to any $\a>0$.

\begin{proposition}
If $X_1,X_2,\dotsc\sim\Normal(0,1)$ i.i.d.\ then with probability~$1$, under the standard normal DPM with $\a = 1$ as defined above, $p(T_n = 1 \mid X_{1:n})$ does not converge to~$1$ as $n\to\infty$.
\end{proposition}
\begin{proof}
Let $n\in\{2,3,\dotsc\}$. 
Let $x_1,\dotsc,x_n\in\R$, $\,A\in\Ac_2(n)$, and $\,a_i =|A_i|$ for $i = 1,2$. Define $s_n =\sum_{j = 1}^n x_j$ and $\,s_{A_i} =\sum_{j \in A_i} x_j$ for $i = 1,2$.
Using Equation~\ref{equation:normal_marginal} and noting that $1/(n +1)\leq 1/(n +2) +1/n^2$, we have
\begin{align*}
\sqrt{n +1}\,\frac{m(x_{1:n})}{p_0(x_{1:n})}
  = \exp\Big(\frac{1}{2}\frac{s_n^2}{n +1}\Big)
\leq\exp\Big(\frac{1}{2}\frac{s_n^2}{n +2}\Big) \exp\Big(\frac{1}{2}\frac{s_n^2}{n^2}\Big).
\end{align*}
The second factor equals $\exp(\frac{1}{2} \bar x_n^2)$, where $\bar x_n =\frac{1}{n}\sum_{j = 1}^n x_j$.
Writing $s_n/(n +2)$ as a convex combination of $s_{A_1}/(a_1+1)$ and $s_{A_2}/(a_2+1)$, by the convexity of $x\mapsto x^2$ the first factor is less or equal to
\begin{align*}
\exp\Big(\frac{1}{2}\frac{s_{A_1}^2}{a_1+1} + \frac{1}{2}\frac{s_{A_2}^2}{a_2+1}\Big)
 = \sqrt{a_1 +1}\sqrt{a_2+1}\,\frac{m(x_{A_1})\,m(x_{A_2})}{p_0(x_{1:n})}.
\end{align*}
Hence,
\begin{align}
\frac{m(x_{1:n})}{m(x_{A_1})\,m(x_{A_2})}
\leq\frac{\sqrt{a_1 +1}\sqrt{a_2+1}}{\sqrt{n +1}} \exp(\tfrac{1}{2} \bar x_n^2).
\label{equation:inequality}
\end{align}
Consequently, we have
\begin{align*}
\frac{p(x_{1:n},T_n = 2)}{p(x_{1:n},T_n = 1)} 
& \eqstack{a} \sum_{A\in \Ac_2(n)} n\,p(A)\,\frac{m(x_{A_1})\,m(x_{A_2})}{m(x_{1:n})} \\
& \geqstack{b}\sum_{A\in \Ac_2(n)} n\,p(A)\,\frac{\sqrt{n+1}}{\sqrt{|A_1|+1}\sqrt{|A_2|+1}}\, \exp(-\tfrac{1}{2}\bar x_n^2)\\
& \geqstack{c}\sum_{\substack{A\in \Ac_2(n):\\|A_1|= 1}} 
    n\,\frac{(n-2)!}{n!\,2!}\,\frac{\sqrt{n+1}}{\sqrt{2}\sqrt{n}}\, \exp(-\tfrac{1}{2}\bar x_n^2)\\
&\geqstack{d}\frac{1}{2\sqrt 2} \, \exp(-\tfrac{1}{2}\bar x_n^2),
\end{align*}
where step (a) follows from applying Equation \ref{equation:x_t} to both numerator and denominator, plus using Equation \ref{equation:partitions} (with $\a = 1$) to see that $p(x_{1:n},T_n = 1) = m(x_{1:n})/n$, step (b) follows from Equation \ref{equation:inequality} above, step (c) follows since all the terms in the sum are nonnegative and $p(A) =(n-2)!/n!\,2!$ when $|A_1|= 1$ (by Equation \ref{equation:partitions}, with $\a = 1$), and step (d) follows since there are $n$ partitions $A\in\Ac_2(n)$ such that $|A_1|= 1$.

If $X_1,X_2,\dotsc\sim\Normal(0,1)$ i.i.d.\ then by the law of large numbers, $\bar X_n=\frac{1}{n}\sum_{j = 1}^n X_j\to 0$ almost surely as $n\to\infty$. Therefore, 
\begin{align*}
p(T_n = 1 \mid X_{1:n})
& = \frac{p(X_{1:n},T_n = 1)}{\sum_{t = 1}^\infty p(X_{1:n},T_n = t)}
\leq \frac{p(X_{1:n},T_n = 1)}{p(X_{1:n},T_n = 1)+p(X_{1:n},T_n = 2)} \\
&\leq\frac{1}{1+\tfrac{1}{2\sqrt 2}\exp(-\tfrac{1}{2}\bar X_n^2)}
\xrightarrow[]{\textup{a.s.}} \frac{1}{1+\tfrac{1}{2\sqrt 2}}.
\end{align*}
Hence, almost surely, $p(T_n = 1 \mid X_{1:n})$ does not converge to $1$.
\end{proof}

Note that the only property of the $\Normal(0,1)$ data distribution that we used was $\bar X_n=\frac{1}{n}\sum_{j = 1}^n X_j\to \E X_1 \in\R$. Thus, we could clearly have let $X_1,X_2,\dotsc$ be i.i.d.\ from any distribution with finite mean, and still $p(T_n = 1 \mid X_{1:n})$ would not converge to~$1$.

\section{Severe inconsistency}
\label{section:proof}

In the previous section, we showed that $p(T_n = 1 \mid X_{1:n})$ does not converge to~$1$ for a standard normal DPM on standard normal data. In this section, we prove that in fact, it converges to~$0$. This vividly illustrates that improper use of DPMs can lead to entirely misleading results. The key step in the proof is an application of Hoeffding's strong law of large numbers for U-statistics. The proof generalizes easily to any $\a>0$.

\begin{theorem}
If $X_1,X_2,\dotsc\sim\Normal(0,1)$ i.i.d.\ then 
$$ p(T_n = 1 \mid X_{1:n})\xrightarrow[]{\textup{Pr}} 0 \spacing\mbox{as }\,n\to\infty$$
under the standard normal DPM with concentration parameter $\a = 1$.
\end{theorem}
\begin{proof}
For $t = 1$ and $t = 2$ define
$$ R_t(X_{1:n}) = n^{3/2}\,\frac{p(X_{1:n},T_n = t)}{p_0(X_{1:n})} . $$
(For general $\a>0$, replace $n^{3/2}$ above by $n^{3/2} \a^{(n)}/n!$.)
Our method of proof is as follows. We will show that
$$R_2(X_{1:n}) \xrightarrow[n\to\infty]{\textup{Pr}} \infty$$
(or in other words, for any $B>0$ we have $\P(R_2(X_{1:n})>B)\to 1$ as $n\to\infty$),
and we will show that $R_1(X_{1:n})$ is bounded in probability:
$$R_1(X_{1:n}) = O_P(1)$$
(or in other words, for any $\epsilon>0$ there exists $B_\epsilon>0$ such that
$\P(R_1(X_{1:n})>B_\epsilon)\leq\epsilon$ for all $n\in\{1,2,\dotsc\}$).
Putting these two together, we will have
\begin{align*}
p(T_n = 1 \mid X_{1:n})
& = \frac{p(X_{1:n},T_n = 1)}{\sum_{t = 1}^\infty p(X_{1:n},T_n = t)}
\leq \frac{p(X_{1:n},T_n = 1)}{p(X_{1:n},T_n = 2)}
=\frac{R_1(X_{1:n})}{R_2(X_{1:n})} \xrightarrow[n\to\infty]{\textup{Pr}} 0.
\end{align*} 

First, let's show that $R_2(X_{1:n})\to\infty$ in probability.
For $S\subset\{1,\dotsc,n\}$ with $|S|\geq 1$, define $h(x_S)$ by
$$h(x_S) =\frac{m(x_S)}{p_0(x_S)}
= \frac{1}{\sqrt{|S|+1}}\exp\Big(\frac{1}{2}\frac{1}{|S|+1}\Big(\sum_{j\in S} x_j\Big)^2\Big), $$
where $m$ is the single-cluster marginal as in Equations \ref{equation:marginal} and \ref{equation:normal_marginal}. 
Note that when $1\leq|S|\leq n-1$, we have $\sqrt n \,h(x_S)\geq 1$. Note also that $\E h(X_S) = 1$ since
$$\E h(X_S) = \int h(x_S)\,p_0(x_S)\,dx_S =\int m(x_S)\,dx_S = 1,$$
using the fact that $m(x_S)$ is a density with respect to Lebesgue measure.
For $k\in\{1,\dotsc,n\}$, define the U-statistics
$$ U_k(X_{1:n})=\frac{1}{{n\choose k}}\sum_{|S|= k} h(X_S)$$
where the sum is over all $S\subset\{1,\dotsc,n\}$ such that $|S|= k$. 
By Hoeffding's strong law of large numbers for U-statistics \cite{Hoeffding_1961},
$$ U_k(X_{1:n}) \xrightarrow[n\to\infty]{\textup{a.s.}} \E h(X_{1:k}) = 1 $$
for any $k\in\{1,2,\dotsc\}$. 
Therefore, using Equations \ref{equation:partitions} and \ref{equation:x_t} we have that for any $K\in\{1,2,\dotsc\}$ and any $n>K$,
\begin{align*}
R_2(X_{1:n}) 
& = n^{3/2}\sum_{A\in\Ac_2(n)} p(A) \,\frac{m(X_{A_1})\,m(X_{A_2})}{p_0(X_{1:n})}\\
& =  n\sum_{A\in\Ac_2(n)} p(A) \sqrt n \,h(X_{A_1})\,h(X_{A_2})\\
&\geq n\sum_{A\in\Ac_2(n)} p(A)\, h(X_{A_1})\\
& = n\sum_{k=1}^{n-1}\sum_{|S|= k} \frac{(k-1)!\,(n-k-1)!}{n!\,2!}\,\,h(X_S) \displaybreak[0]\\
& =\sum_{k=1}^{n-1} \frac{n}{2k(n-k)}\,\frac{1}{{n\choose k}} \sum_{|S|= k} h(X_S)\\
& =\sum_{k=1}^{n-1} \frac{n}{2k(n-k)}\,\, U_k(X_{1:n})\\
&\geq\sum_{k=1}^K \frac{n}{2k(n-k)}\,\, U_k(X_{1:n})\\
&\xrightarrow[n\to\infty]{\textup{a.s.}} \sum_{k=1}^K \frac{1}{2k}
=\frac{H_K}{2} > \frac{\log K}{2}
\end{align*}
where $H_K$ is the $K^\mathrm{th}$ harmonic number, and the last inequality follows from the standard bounds \cite{Graham_1989} on harmonic numbers: $\log K<H_K\leq\log K+1$.
Hence, for any $K$,
$$\liminf_{n\to\infty} R_2(X_{1:n})>\frac{\log K}{2} \spacing\mbox{almost surely,}$$
and it follows easily that
$$ R_2(X_{1:n})\xrightarrow[n\to\infty]{\textup{a.s.}} \infty. $$
Convergence in probability is implied by almost sure convergence.

Now, let's show that $R_1(X_{1:n}) = O_P(1)$. By Equation \ref{equation:partitions}, $p(A) = 1/n$ when $A =(\{1,\dotsc,n\})$. Using this along with Equations \ref{equation:x_t} and \ref{equation:normal_marginal}, we have
\begin{align*}
R_1(X_{1:n}) &= n^{3/2}\,\frac{p(X_{1:n},T_n = 1)}{p_0(X_{1:n})} 
=\sqrt n\,\frac{m(X_{1:n})}{p_0(X_{1:n})}\\
& =\frac{\sqrt n}{\sqrt{n +1}}\,\exp\Big(\frac{1}{2}\frac{n}{n +1}\,\Big(\frac{1}{\sqrt n}\sum_{i = 1}^n X_i\Big)^2\Big)
 \leq \exp (Z_n^2/2)
\end{align*}
where $Z_n=(1/\sqrt n)\sum_{i = 1}^n X_i\sim\Normal(0,1)$ for each $n\in\{1,2,\dotsc\}$. Since $Z_n = O_P(1)$ then we conclude that $R_1(X_{1:n}) = O_P(1)$.
This completes the proof. 
\end{proof}

\section*{Acknowledgments}
We would like to thank Stu Geman for raising this question. 
This research was supported in part by the National Science Foundation under grant DMS-1007593 and the Defense Advanced Research Projects Agency under contract FA8650-11-1-715.

\bibliography{references}

\providecommand{\bysame}{\leavevmode\hbox to3em{\hrulefill}\thinspace}
\providecommand{\MR}{\relax\ifhmode\unskip\space\fi MR }
\providecommand{\MRhref}[2]{%
  \href{http://www.ams.org/mathscinet-getitem?mr=#1}{#2}
}
\providecommand{\href}[2]{#2}
\begin{thebibliography}{10}

\bibitem{Escobar_1995}
M.D. Escobar and M.~West, \emph{Bayesian density estimation and inference using
  mixtures}, Journal of the American Statistical Association \textbf{90}
  (1995), no.~430, 577--588.

\bibitem{Escobar_1998}
\bysame, \emph{Computing nonparametric hierarchical models}, Practical
  Nonparametric and Semiparametric {B}ayesian Statistics (D.~Dey,
  P.~M\"{u}ller, and D.~Sinha, eds.), Springer-Verlag, New York, 1998,
  pp.~1--22.

\bibitem{Ferguson_1983}
T.S. Ferguson, \emph{Bayesian density estimation by mixtures of normal
  distributions}, Recent Advances in Statistics (M.~H. Rizvi, J.~Rustagi, and
  D.~Siegmund, eds.), Academic Press, 1983, pp.~287--302.

\bibitem{Ghosh_2003}
J.K. Ghosh and R.V. Ramamoorthi, \emph{Bayesian {N}onparametrics},
  Springer-Verlag, New York, 2003.

\bibitem{Graham_1989}
R.L. Graham, D.E. Knuth, and O.~Patashnik, \emph{Concrete {M}athematics},
  Addison-Wesley, 1989.

\bibitem{Green_2001}
P.J. Green and S.~Richardson, \emph{Modeling heterogeneity with and without the
  {D}irichlet process}, Scandinavian Journal of Statistics \textbf{28} (2001),
  no.~2, 355--375.

\bibitem{Hoeffding_1961}
W.~Hoeffding, \emph{The strong law of large numbers for {U}-statistics},
  Institute of Statistics, Univ. of N. Carolina, Mimeograph Series \textbf{302}
  (1961).

\bibitem{Huelsenbeck_2007}
J.P. Huelsenbeck and P.~Andolfatto, \emph{Inference of population structure
  under a {D}irichlet process model}, Genetics \textbf{175} (2007), no.~4,
  1787--1802.

\bibitem{Lartillot_2004}
N.~Lartillot and H.~Philippe, \emph{A {B}ayesian mixture model for across-site
  heterogeneities in the amino-acid replacement process}, Molecular Biology and
  Evolution \textbf{21} (2004), no.~6, 1095--1109.

\bibitem{Medvedovic_2002}
M.~Medvedovic and S.~Sivaganesan, \emph{Bayesian infinite mixture model based
  clustering of gene expression profiles}, Bioinformatics \textbf{18} (2002),
  no.~9, 1194--1206.

\bibitem{Nobile_1994}
A.~Nobile, \emph{Bayesian analysis of finite mixture distributions}, Ph.D.
  thesis, Department of Statistics, Carnegie Mellon University, Pittsburgh, PA,
  1994.

\bibitem{Nobile_2007}
A.~Nobile and A.T. Fearnside, \emph{Bayesian finite mixtures with an unknown
  number of components: The allocation sampler}, Statistics and Computing
  \textbf{17} (2007), no.~2, 147--162.

\bibitem{Onogi_2011}
A.~Onogi, M.~Nurimoto, and M.~Morita, \emph{Characterization of a {B}ayesian
  genetic clustering algorithm based on a {D}irichlet process prior and
  comparison among {B}ayesian clustering methods}, BMC Bioinformatics
  \textbf{12} (2011), no.~1, 263.

\bibitem{Otranto_2002}
E.~Otranto and G.M. Gallo, \emph{A nonparametric {B}ayesian approach to detect
  the number of regimes in {M}arkov switching models}, Econometric Reviews
  \textbf{21} (2002), no.~4, 477--496.

\bibitem{Richardson_1997}
S.~Richardson and P.J. Green, \emph{On {B}ayesian analysis of mixtures with an
  unknown number of components}, Journal of the Royal Statistical Society.
  Series B \textbf{59} (1997), no.~4, 731--792.

\bibitem{West_1994}
M.~West, P.~M\"{u}ller, and M.D. Escobar, \emph{Hierarchical priors and mixture
  models, with application in regression and density estimation}, Institute of
  Statistics and Decision Sciences, Duke University, 1994.

\bibitem{Wu_2010}
Y.~Wu and S.~Ghosal, \emph{The {L}$_1$-consistency of {D}irichlet mixtures in
  multivariate {B}ayesian density estimation}, Journal of Multivariate Analysis
  \textbf{101} (2010), no.~10, 2411--2419.

\bibitem{Xing_2006}
E.P. Xing, K.A. Sohn, M.I. Jordan, and Y.W. Teh, \emph{Bayesian
  multi-population haplotype inference via a hierarchical {D}irichlet process
  mixture}, Proceedings of the 23rd International Conference on Machine
  Learning, 2006, pp.~1049--1056.

\end{thebibliography}
\bibliographystyle{amsplain}

\end{document}